\documentclass[envcountset]{llncs}

\usepackage{mathtools}

\usepackage{tikz}
\usetikzlibrary{arrows}
\usepackage{amsmath}
\usepackage{amssymb}
\usepackage{caption}
\usepackage{float}

\newcommand{\tuple}[1]{\vec{#1}}
\newcommand {\indep}[3] {#2 ~\bot_{#1}~ #3}
\newcommand {\indepc}[2] {#1 ~\bot~ #2}

\newcommand{\Dom}{\textrm{Dom}}

\newcommand{\Fr}{\textrm{Fr}}

\newcommand{\C}{\mathbb{C}}
\newcommand{\Q}{\mathbb{Q}}

\newcommand{\Po}{\mathcal{P}}

\newcommand{\M}{\mathfrak{M}}
\newcommand{\A}{\mathfrak{A}}

\newcommand{\F}{\mathcal{F}}
\newcommand{\on}{\exists}
\newcommand{\ja}{\wedge}
\newcommand{\tai}{\vee}

\renewcommand{\a}{\alpha}
\renewcommand{\b}{\beta}

\newcommand{\e}{\varepsilon}

\def\dep{=\!\!}

\newcommand{\ESOarity}[1]{{\rm ESO}({#1}\mbox{\rm-ary})}

\newcommand{\ESOfvar}[1]{{\rm ESO}_f({#1}\forall)}

\newcommand{\PTIME}{{\rm PTIME}}
\newcommand{\deplogic}{\FO (\dep(\ldots ))}
\newcommand{\deparity}[1]{\deplogic ({#1}\mbox{\rm-dep})}
\newcommand{\depforall}[1]{\deplogic ({#1}\forall)}
\newcommand{\ind}[1]{\FO (\bot_{\rm c})({#1}\mbox{\rm-ind})}

\newcommand{\inc}[1]{\FO (\subseteq)({#1}\mbox{\rm-inc})}
\newcommand{\incforall}[1]{\FO (\subseteq) ({#1}\forall)}
\newcommand{\indforall}[1]{\FO (\bot_{\rm c}) ({#1}\forall)}
\newcommand{\indNRforall}[1]{\FO (\bot) ({#1}\forall)}

\newcommand{\indNRincforall}[1]{\FO (\bot,\subseteq) ({#1}\forall)}
\newcommand{\indlogic}{\FO (\bot_{\rm c})}
\newcommand{\inclogic}{\FO (\subseteq)}
\newcommand{\indNRlogic}{\FO (\bot)}
\newcommand{\indNRinclogic}{\FO (\bot,\subseteq)}

\newcommand{\sub}{\subseteq}
\newcommand{\EDGE}[1] {\textrm{E{\small DGE}}_{#1}}
\newcommand{\Edge}[2] {\textrm{E{\small DGE}}_k(#1, #2)}
\newcommand{\auto}[1]{{\rm auto}{(#1)}}
\newcommand{\row}[1]{{\rm row}{(#1)}}

\newcommand{\mi}[1]{{\rm mid}{(#1)}}
\newcommand{\swap}[1]{{\rm swap}{(#1)}}
\newcommand{\atomss}{\sub, \dep(\ldots),\bot_{\rm c}, \bot}
\newcommand{\atoms}{\sub, \dep(\ldots),\bot_{\rm c}}
\newcommand{\Lo}{\mathcal{L}}
\renewcommand{\C}{\mathcal{C}}
\newcommand{\LFP}{\rm LFP}
\newcommand{\GFP}{\rm GFP}
\newcommand{\IFP}{\rm IFP}
\newcommand{\TC}{\rm TC}
\newcommand{\PFP}{\rm PFP}
\newcommand{\LOGSPACE}{\rm LOGSPACE}
\newcommand{\PSPACE}{\rm PSPACE}
\newcommand{\PFPary}[1]{{\PFP}^{#1}}
\newcommand{\IFPary}[1]{{\IFP}^{#1}}
\newcommand{\LFPary}[1]{{\LFP}^{#1}}
\newcommand{\TCary}[1]{{\TC}^{#1}}
\def\dep{=\!\!}
\newcommand{\FO}{{\rm FO}}

\newcommand{\ESO}{{\rm ESO}}

\tikzstyle{my loop}=[->,to path={
  .. controls +(315:0.8) and +(225:0.8) .. (\tikztotarget) \tikztonodes}]
\tikzstyle{my loops}=[->,to path={
  .. controls +(45:0.8) and +(135:0.8) .. (\tikztotarget) \tikztonodes}]
\tikzstyle{my loopss}=[->,to path={
  .. controls +(45:0.8) and +(315:0.8) .. (\tikztotarget) \tikztonodes}]

\begin{document}

\author{Miika Hannula\thanks{Department of Mathematics and Statistics, University of Helsinki, Finland. \texttt{miika.hannula@helsinki.fi}}}
\institute{University of Helsinki}
\title{Hierarchies in inclusion logic  with lax semantics\thanks{The author was supported by grant 264917 of the Academy of Finland.}}

\maketitle

\begin{abstract}
We study the expressive power of fragments of inclusion logic under the so-called lax team semantics. The fragments are defined either by restricting the number of universal quantifiers or the arity of inclusion atoms in formulae. In  case of universal quantifiers, the corresponding hierarchy collapses at the first level. Arity hierarchy is shown to be strict by relating the question to the study of arity hierarchies in fixed point logics.
\end{abstract}

\section{Introduction}
In this article we study the expressive power of inclusion logic ($\inclogic$) \cite{galliani12} in the lax team semantics setting. Inclusion logic is a variant of dependence logic ($\deplogic$) \cite{vaananen07} which extends first-order logic with dependence atoms
$$\dep(x_1, \ldots ,x_n)$$
expressing that the values of $x_n$ depend functionally on the values of $x_1, \ldots ,x_n$. Inclusion logic, instead, extends first-order logic with inclusion atoms 
$$\tuple x \sub \tuple y$$ 
which express that the set of values of $\tuple x$ is included in the set of the values of $\tuple y$. We study the expressive power of two syntactic fragments of inclusion logic under the lax team semantics. These two fragments, $\incforall{k}$ and $\inc{k}$, are defined by restricting the number of universal quantifiers or the arity of inclusion atom to $k$, respectively. We will show that $\incforall{k}$ captures $\inclogic$ already with $k=1$ and that the fragments $\inc{k}$ give rise to an infinite expressivity hierarchy.

Since the introduction of dependence logic in 2007, many interesting variants of it have been introduced. One reason for this orientation is the semantical framework that is being used. Team semantics, introduced by Hodges in 1997 \cite{hodges97}, provides a natural  way to extend first-order logic with many different kind of dependency notions. Although many of these notions have been  extensively studied in database theory since the 70s,  with team semantics the novelty comes from the fact that  also interpretations for logical connectives and quantifiers are provided.

In expressive power $\deplogic$ is equivalent to existential second-order logic ($\ESO$) \cite{vaananen07}. For some variants of $\deplogic$, the correspondence to  $\ESO$ does not hold or it can depend on which version of team semantics is being used. For instance, $\inclogic$ corresponds in expressive power to $\ESO$ if we use the so-called strict team semantics \cite{galhankon13}. Under the lax team semantics, $\inclogic$ corresponds to greatest fixed point logic ($\GFP$) \cite{gallhella13} which captures $\PTIME$ over finite ordered models. In $\deplogic$  no separation between the strict and the lax version of team semantics exists since dependence atoms satisfy the so-called downvard closure property. In the following we briefly list some complexity theoretical aspects of $\deplogic$ and its variants.
\begin{itemize}
\item $\deplogic$ extended with the so-called intuitionistic implication $\rightarrow$ (introduced in \cite{abramsky09}) increases the expressive power of $\deplogic$ to full second-order logic \cite{yang13}.
\item The model checking problem of $\deplogic$, and many of its variants, was recently shown to be $\rm NEXPTIME$-complete. Moreover, for any variant of $\deplogic$ whose atoms are $\PTIME$-computable, the corresponding model checking problem is contained in $\rm NEXPTIME$ \cite{gradel13}.
\item The non-classical interpretation of disjunction in $\deplogic$ has the effect that the model checking problems of $\phi_1 := \hspace{2mm}\dep(x,y) \hspace{1mm} \tai \dep(u,v)$ and $\phi_2 :=\hspace{2mm} \dep(x,y) \hspace{1mm} \tai \dep(u,v) \hspace{1mm}\tai \dep(u,v)$ are already $\rm NL$-complete and $\rm NP$-complete, respectively \cite{kontinen_ja10}.
\end{itemize}
This article pursues the line of study taken in \cite{durand12} and \cite{galhankon13} where syntactical fragments of dependence and independence logic ($\indlogic$) were investigated, respectively. $\indlogic$ extends first-order logic by  conditional independence atoms
$$\indep{\tuple x}{\tuple y}{\tuple z}$$
with the informal meaning that the values of $\tuple y$ and $\tuple z$ are independent of each other, given any value of $\tuple x$. As $\inclogic$, also $\indlogic$ does not have downvard closure and is sensitive to the choice between the lax and the strict version of team semantics. For a sequence of atoms $\C$, we use $\FO(\C)$ to denote the logic obtained by adding the atoms listed in $\C$ to first-order logic. $\FO(\C)(k\forall)$  denotes the sentences of $\FO(\C)$ in which at most $k$ variables are universally quantified. In \cite{durand12} it was shown that 
$$\depforall{k} \leq \ESOfvar{k} \leq \depforall{2k}$$ where $\ESOfvar{k}$ denotes the skolem normal form $\ESO$ sentences in which at most $k$ universally quantified first-order variables appear. In \cite{galhankon13} it was shown that (under the lax team semantics)
\begin{itemize}
\item $\indNRforall{2} = \indNRlogic$ and
\item $\indNRincforall{1} = \indNRinclogic$
\end{itemize} 
where $\indNRlogic$ is the sublogic of $\indlogic$ allowing only so-called pure independence atoms $\indepc{\tuple x}{\tuple y}$. Moreover, it is known that $\indNRlogic$ is equivalent in expressive power to $\indNRinclogic$ and $\indlogic$ \cite{galliani12,vaananen13}.

Also arity fragments of $\FO(\C)$ were defined. By $\FO(\C)(k \mbox{\rm -dep})$ we denote the sentences of $\FO(\C)$ in which dependence atoms of the form $\dep(x_1, \ldots ,x_{n+1})$ with $n \leq k$ may appear. $\FO(\C)(k \mbox{\rm -ind})$ denotes the sentences of $\indlogic$ in which independence atoms containing at most $k+1$ different variables may appear. It was shown in \cite{durand12,galhankon13} that (under the lax team semantics)
$$\ESOarity{k}=\deparity{k} = \ind{k}$$ 
where $\ESOarity{k}$ denote the sentences of $\ESO$ in which the quantified functions and relations have arity at most $k$. This yields an infinite arity hierarchy for both $\deplogic$ and $\indlogic$ since the property "$R$ is even" is definable in $\ESOarity{k}$ but not in $\ESOarity{k-1}$, for $k$-ary $R$ \cite{ajtai83}.

The main contribution of this article is to show that arity fragments of inclusion logic also give rise to an infinite expressivity hierarchy. We let $\FO(\C)(k \mbox{\rm -inc})$ denote the $\FO(\C)$ sentences in which at most $k$-ary inclusion atoms (i.e. atoms of the form $\tuple x \sub \tuple y$ where $|\tuple x|=|\tuple y| \leq k$) may appear. For proving the claim, we define, for each $k\geq 2$, a graph property which is definable in $\inc{k}$ but not in $\inc{k-1}$. The non-definability part of the proof will be based on Martin Grohe's work in fixed point logics in \cite{Grohe96} where analogous results for $\TC$, $\LFP$, $\IFP$ and $\PFP$ were proved. We will also give a negative answer to the open question presented in \cite{galhankon13}; that was, whether the fragments $\incforall{k}$ give rise to an infinite expressivity hierarchy. This will be done by showing that $\incforall{1}=\inclogic$. However, if the strict version of team semantics is used, then we obtain $\incforall{k} < \incforall{k+1}$ \cite{hankon14}.

\section{Preliminaries}
In this section we give a short introduction to dependence, independendence and inclusion logic.

\subsection{Notation}
Unless otherwise stated, we use $x_1,x_2,\ldots$ 
to denote variables and $t_1,t_2,\ldots$ 
to denote terms. Analogously, bolded versions $\tuple x_1,\tuple x_2,\ldots$ and $\tuple t_1,\tuple t_2,\ldots$ are used to denote tuples of variables and tuples of terms, respectively. For tuples $\tuple a$ and $\tuple b$, we write $\tuple a\tuple b$ for  the concatenation of the tuples. If $f$ is a unary function and $(x_1, \ldots ,x_n)$ is a sequence listing members of $\Dom(f)$, then we write $f(x_1, \ldots ,x_n)$ for $(f(x_1), \ldots ,f(x_n))$.

\subsection{Inclusion logic}
The syntax of $\inclogic$ is obtained by adding  inclusion atoms to the syntax of  first-order logic. 
\begin{definition}\label{def1}
$\inclogic$ is defined by the following grammars. Note that in an inclusion atom $\tuple x_1 \sub \tuple x_2$, the tuples $\tuple x_1$ and $\tuple x_2$ must be of the same length.
$$\phi::= \tuple x_1 \sub \tuple x_2 \mid t_1 = t_2 \mid \neg t_1 = t_2 \mid R(\tuple t) \mid \neg R(\tuple t) \mid (\phi \tai \psi) \mid (\phi \tai \psi) \mid \forall x \phi \mid \on x \phi.$$
\end{definition}
$\deplogic$ and $\indlogic$ are obtained from Definition \ref{def1} by replacing inclusion atoms $\tuple x_1 \sub \tuple x_2$ with dependence atoms $\dep(\tuple x_1, x_2)$ and conditional independence atoms $\indep{\tuple x_1}{\tuple x_2}{\tuple x_3}$, respectively. Pure independence logic $\indNRlogic$ is a fragment of $\indlogic$ where only pure independence atoms $\indepc{\tuple x_1}{\tuple x_2}$ (i.e. atoms of the form $\indep{\emptyset}{\tuple x_1}{\tuple x_2}$) may appear. Also, for any sequence $\mathcal{C}$ of dependency atoms of $\{\sub, \dep(\ldots ), \bot_{\rm c},\bot \}$ we use $\FO(\mathcal{C})$ to denote the logic obtained from Definition \ref{def1} by replacing inclusion atoms with atoms listed in $\mathcal{C}$.

In order to define semantics for these logics, we need to define the concept of a \emph{team}. Let $\M$ be a model with the domain $M$. We assume that all our models have at least two elements.\footnote{This assumption is needed in Theorem \ref{thm3}.} An \emph{assignment} over $\M$ is a finite function that maps variables to elements of $M$. A \emph{team} $X$ of $M$ with the domain $\Dom(X) = \{x_1, \ldots ,x_n\}$ is a set of assignments from $\Dom(X)$ into $M$.
If $X$ is a team of $M$ and $F: X \rightarrow \Po(M)\setminus \{\emptyset\}$, then we use $X[F/x]$ to denote the team $\{s(a/x) \mid s\in X, a \in F(s)\}$ and $X[M/x]$ for $\{s(a/x) \mid s\in X, a \in M\}$. Also one should note that if $s$ is an assignment, then $\M \models_s \phi$ refers to Tarski semantics and $\M \models_{\{s\}} \phi$ refers to team semantics.

\begin{definition}\label{def2}
For a model $\M$, a team $X$ and a formula in $\FO(\sub, \dep(\ldots ),\bot_{\rm c})$, the satisfaction relation $\M \models_X \phi$ is defined as follows:
\begin{itemize}
\item $\M \models_X \a \Leftrightarrow \forall s\in X(\M \models_s \a)$, when $\a$ is a first-order literal,
\item $\M \models_X \tuple x_1 \sub \tuple x_2 \Leftrightarrow \forall s \in X\on s' \in X\big (s(\tuple x_1) = s'(\tuple x_2)\big )$,
\item $\M \models_X \indep{\tuple x_1}{\tuple x_2}{\tuple x_3} \Leftrightarrow \forall s,s' \in X\big (s(\tuple x_1) =s'(\tuple x_2) \Rightarrow \\\on s'' \in X(s''(\tuple x_1)=s(\tuple x_1), s''(\tuple x_2) = s(\tuple x_2),s''(\tuple x_3)= s'(\tuple x_3))\big )$,
\item $\M \models_X \dep(\tuple x_1, x_2) \Leftrightarrow \forall s,s'\in X\big (s(\tuple x_1) = s'(\tuple x_1) \Rightarrow  s(x_2) =  s'(x_2)\big )$,
\item $\M \models_X \phi \ja \psi \Leftrightarrow \M \models_X \phi \textrm{ and } \M \models_X \psi$,
\item $\M \models_X \phi \tai \psi \Leftrightarrow\M \models_Y \phi \textrm{ and } \M \models_Z \psi$, for some $Y\cup Z = X$,
\item $\M \models_X \on x \phi \Leftrightarrow  \M \models_{X[F/x]} \phi$, for some $F: X \rightarrow \Po(M)\setminus \{\emptyset\}$,
\item $\M \models_X \forall x \phi \Leftrightarrow  \M \models_{X[M/x]} \phi$.
\end{itemize}
If $\M \models_X \phi$, then we say that $X$ satisfies $\phi$ in $\M$. If $\phi$ is a sentence and $\M \models_{\{\emptyset\}}\phi $\footnote{$\{\emptyset\}$ denotes the team that consists of the empty assignment.}, then we say that $\phi$ is true in $\M$, and write $\M \models \phi$.
\end{definition}
Note that in Definition \ref{def2}, we obtain the \emph{lax} version of team semantics. The \emph{strict} version of team semantics is defined otherwise as in Definition \ref{def2} except that only disjoint subteams are allowed to witness split disjunction and existential quantification ranges over $M$ instead of non-empty subsets of $M$. (See \cite{galliani12} for more information.)

First-order formulae are \emph{flat} in the following sense (the proof is a straightforward structural induction). 
\begin{theorem}[Flatness]\label{flatness}
For a model $\M$, a first-order formula $\phi$ and a team $X$, the following are equivalent:
\begin{itemize}
\item $\M \models_X \phi$,
\item $\M \models_{\{s\}} \phi $ for all $s\in X$,
\item $\M \models_{s} \phi $ for all $s \in X$.
\end{itemize}
\end{theorem}
By $\Fr(\phi)$ we denote the set of variables that appear free in $\phi$. If $X$ is a team and $V$ a set of variables, then $X \upharpoonright V$ denotes the team $\{s\upharpoonright V \mid s \in X\}$. Now, all formulae satisfy the following \emph{locality} property. Note that this is not true under the strict team semantics.
\begin{theorem}[Locality]\label{locality}
Let $\M$ be a model, $X$ be a team, $\phi \in \FO(\sub, \dep(\ldots ),\bot_{\rm c})$ and $V$ a set of variables such that $\Fr(\phi) \sub V \sub \Dom(X)$. Then
$$\M \models_X \phi \Leftrightarrow \M \models_{X\upharpoonright V} \phi.$$
\end{theorem}
We say that formulae $\phi,\psi \in  \FO(\sub, \dep(\ldots ),\bot_{\rm c})$ are \emph{logically equivalent}, written $\phi \equiv \phi'$, if for all models $\M$ and teams $X$ such that $\Fr(\phi) \cup \Fr(\psi) \sub \Dom(X)$,
$$\M \models_X \phi \Leftrightarrow \M \models_X \psi.$$
We obtain the following normal form theorem.
\begin{theorem}[\cite{galhankon13}]\label{thm3}
Any formula $\phi \in \FO (\atoms)$ is logically equivalent to a formula $\phi'$ such that
\begin{itemize}
\item $\phi'$ is of the form $Q^1 x_1 \ldots Q^n x_n \psi$ where $\psi$ is quantifier-free,
\item any literal or dependency atom which occurs in $\phi'$ occurred already in $\phi$,
\item the number of universal quantifiers in $\phi'$ is the same as the number of universal quantifiers in $\phi$.
\end{itemize}
\end{theorem}
For logics $\mathcal{L}$ and $\mathcal{L}'$, we write $\Lo \leq \Lo'$, if for every signature $\tau$,  every $\Lo[\tau]$-sentence is logically equivalent to some $\Lo'[\tau]$-sentence. We write $\Lo \leq_{\mathcal{O}} \Lo'$ if $\Lo \leq \Lo'$ is true in finite linearly ordered models. Equality and inequality relations are obtained from $\leq$ naturally.
We end this section with the following list of theorems characterizing the expressive powers of our logics.
\begin{theorem}[\cite{vaananen07,gradel10,vaananen13,gallhella13}]\label{list}
\begin{itemize}
\item $\deplogic = \FO(\bot_{\rm c})= \FO(\bot)= \ESO$,
\item $\inclogic = \rm GFP$.
\end{itemize}
\end{theorem}

\section{Hierarchies in Inclusion Logic}
In this section we  consider universal and arity fragments of inclusion logic. In Subsection \ref{sub1} we will define these fragments and also concepts of strictness and collapse of a hierarchy. In Subsection \ref{sub2} and \ref{sub3} we will prove collapse of the universal hierarchy and strictness of the arity hierarchy, respectively. 

\subsection{Syntactical Fragments}\label{sub1}

\begin{definition} Let $\mathcal{C}$ be a list of dependencies of $\{\atomss\}$. Then universal and arity fragments of $\FO(\mathcal{C})$ are defined as follows:
\begin{itemize}
\item $\FO(\mathcal{C})(k\forall)$ is the class of $\FO(\mathcal{C})$ formulae in which at most $k$ universal quantifiers may appear,
\item $\FO(\mathcal{C})({k}\mbox{\rm-inc})$ is the class of $\FO(\mathcal{C})$ formulae in which inclusion atoms of the form $\tuple x_1 \sub \tuple x_2$ where $\tuple x_1$ and $\tuple x_2$ are sequences of length at most $k$, may appear,
\item $\FO(\mathcal{C})({k}\mbox{\rm-dep})$ is the class of $\FO(\mathcal{C})$ formulae in which dependence atoms of the form $\dep(\tuple x_1 , x_2)$ where $\tuple x_1 x_2$ is a sequence of length at most $k+1$, may appear,
\item $\FO(\mathcal{C})({k}\mbox{\rm-ind})$ is the class of $\FO(\mathcal{C})$ formulae in which conditional independence atoms of the form $\indep{\tuple x_1}{\tuple x_2}{\tuple x_3}$ where $\tuple x_1 \tuple x_2 \tuple x_3$ is a sequence listing at most $k+1$ distinct variables, may appear.
\end{itemize}
\end{definition}
For a sequence of logics $(\Lo_k)_{k \in \mathbb{N}}$, we say that the $\Lo_k$-hierarchy collapses at level $m$ if $\Lo_m= \bigcup_{k\in \mathbb{N}} \Lo_{k}$. If the hierarchy does not collapse at any level, then we say that it is strict.

As mentioned before, we will show that the $\incforall{k}$-hierarchy collapses already at level $1$ but $\inc{k}$ forms a strict hierarchy which holds already in finite models. 

\subsection{Collapse of the Universal Hierarchy}\label{sub2}
We will first show that the universal hierarchy of inclusion logic collapses. This is done by introducing a translation where all universal quantifiers are removed, and new existential quantifiers, new inclusion atoms and one new universal quantifier are added. The translation will hold already at the level of formulae.

\begin{theorem}
$\incforall{1}=\inclogic$.
\end{theorem}
\begin{proof}
Let $\phi \in \inclogic$ be a formula. We will define a $\phi' \in \incforall{1}$ such that $\phi \equiv \phi'$. By Theorem \ref{thm3} we may assume that $\phi$ is of the form
$$Q^1 x_1 \ldots Q^n x_n \theta$$
where $\theta$ is quantifier-free. We  let
$$\phi':=\on x_1 \ldots \on x_n \forall y (\bigwedge_{\substack{1 \leq i \leq n \\Q^i = \forall}} \tuple z x_1 \ldots x_{i-1} y \sub  \tuple z x_1 \ldots x_{i-1}x_i \ja \theta)
$$
where $\tuple z$ lists $\Fr(\phi)$. Let now $\M$ be a model and $X$ a team such that $ \Fr(\phi)\sub \Dom(X)$; we show that $\M \models_X \phi \Leftrightarrow \M \models_X \phi'$. By Theorem \ref{locality} we may assume without loss of generality that $ \Fr(\phi)= \Dom(X)$. Assume first that $\M \models_X \phi$ when there are, for $1 \leq i \leq n$, functions $F_i : X[F_1/x_1]\ldots [F_{i-1}/x_{i-1}] \rightarrow \Po (M) \setminus \{\emptyset\}$ such that $F_i(s) = M$ if $Q^i = \forall$, and $\M \models_{X'} \theta$
where 
$$X':= X[F_1/x_1]\ldots [F_{n}/x_{n}] .$$
For $\M \models_X \phi'$, it suffices to show that  
\begin{equation}\label{iso}
\M \models_{X'[M/y]} \bigwedge_{\substack{1 \leq i \leq n \\Q^i = \forall}} \tuple z x_1 \ldots x_{i-1} y \sub  \tuple z x_1 \ldots x_{i-1}x_i \ja \theta.
\end{equation}
By Theorem \ref{locality} $\M \models_{X'[M/y]} \theta$, so it suffices to consider only the new inclusion atoms of \eqref{iso}.
So let 1 $\leq i \leq n $ be such that $Q^i= \forall$ and let $s \in X'[M/y]$; we need to find a $s' \in X'[M/y]$ such that $s (\tuple z x_1 \ldots x_{i-1} y ) = s'(  \tuple z x_1 \ldots x_{i-1}x_i)$.  Now, since $Q^i = \forall$, we note that $s(s(y)/x_i) \in X' \upharpoonright (\Fr(\phi) \cup \{x_1, \ldots ,x_i\})$. Therefore we may choose $s'$ to be any extension of $s(s(y)/x_i)$ in $X'[M/y]$.

For the other direction, assume that $\M \models_{X} \phi'$. Then for $1 \leq i \leq n$, there are functions $F_i : X[F_1/x_1]\ldots [F_{i-1}/x_{i-1}] \rightarrow \Po (M) \setminus \{\emptyset\}$ such that \eqref{iso} holds, 
for $X':= X[F_1/x_1]\ldots [F_{n}/x_{n}].$ By Theorem \ref{locality} $\M \models_{X'} \theta$, so it suffices to show that, for all $1 \leq i \leq n$ with $Q^i = \forall$, $F_i$ is the constant function which maps assignments to $M$. So let $i$ be of the above kind, and let $s \in  X[F_1/x_1]\ldots [F_{i-1}/x_{i-1}]$ and $a \in M$. We need show that $s(a/x_i) \in X[F_1/x_1]\ldots [F_{i}/x_{i}]$. First note that since $y$ is universally quantified, $s(a/y)$ has an extension $s_0$ in $X'[M/y]$. Therefore, by \eqref{iso}, there is $s_1 \in X'[M/y]$ such that $s_0 (\tuple z x_1 \ldots x_{i-1} y ) = s_1(  \tuple z x_1 \ldots x_{i-1}x_i)$. Since now $s_1$ agrees with $s$ in $\Fr(\phi)\cup \{x_1, \ldots ,x_{i-1}\}$ and maps $x_i$ to $a$, we obtain that
$$s(a/x_i) = s_1 \upharpoonright ( \Fr(\phi) \cup \{x_1, \ldots ,x_i\}) \in X[F_1/x_1]\ldots [F_{i}/x_{i}].$$
\qed
\end{proof}

\subsection{Strictness of the Arity Hierarchy}\label{sub3}

In this section we will show that the following strict arity hierarchy holds (already in finite models).

\begin{theorem}\label{hierarkia}
For $k \geq 2$, $\inc{k-1} < \inc{k}$.
\end{theorem}
For proving this, we will use the earlier work of Grohe in \cite{Grohe96} where an analogous result was proved for $\TC$, $\LFP$, $\IFP$ and $\PFP$. More precisely, it was shown that, for $k \geq 2$,
\begin{equation}\label{fix}
\TCary{k} \not\leq \PFPary{k-1}
\end{equation}
where the superscript part gives the maximum arity allowed for the fixed point operator. Since $\TCary{k} \leq \LFPary{k} \leq \IFPary{k} \leq \PFPary{k}$, a strict arity hierarchy is obtained for each of these logics.

We start by fixing $\tau$ as the signature consisting of one binary relation symbol $E$ and $2k$ constant symbols $b_1, \ldots ,b_k,c_1, \ldots ,c_k$. The idea is to present a $\inc{k}[\tau]$-definable graph property, and show that it is not definable in $\inc{k-1}[\tau]$.
This graph property will actually be negated version of the one that separates the fragments in \eqref{fix}.
For this, we first define a first-order  formula indicating that the $k$-tuples $\tuple x$ and $\tuple y$ form a $2k$-clique in a graph. Namely, we define $\Edge{\tuple x}{\tuple y}$ as follows:
$$\Edge{\tuple x}{\tuple y}:=\bigwedge_{1 \leq i ,j \leq k} E(x_i,y_j) \ja \bigwedge_{1 \leq i \neq j \leq k}( E(x_i,x_j) \ja E(y_i,y_j)).$$
Then the non-trivial part is to show that negation of the transitive closure formula $[TC_{\tuple x , \tuple y} \EDGE{k}](\tuple b , \tuple c)$ is not definable in $\inc{k-1}[\tau]$. It is definable in $\inc{k}[\tau]$ by the following theorem.

\begin{theorem}[\cite{galliani12}]\label{pietron}
Let $\psi(\tuple x, \tuple y)$ be any first-order formula, where $\tuple x$ and $\tuple y$ are tuples of disjoint variables of the same arity. Furthermore, let $\psi'(\tuple x, \tuple y)$ be the result of writing $\neg \psi (\tuple x ,\tuple y)$ in negation normal form. Then, for all suitable models $\M$ and all suitable pairs $\tuple b$, $\tuple c$ of constant term tuples of the model, 
$$\M \models \phi \Leftrightarrow \M \models\neg[TC_{\tuple x , \tuple y} \psi](\tuple b , \tuple c) ,
$$
for $\phi$ defined as
$$\on \tuple z ( \tuple b \sub \tuple z \ja \tuple z \neq \tuple c \ja \forall \tuple w ( \psi' ( \tuple z ,\tuple w) \tai \tuple w \sub \tuple z ) ).$$
\end{theorem}
Note that $\phi$ is not yet of the right form since Definition \ref{def1} does not allow terms to appear in inclusion atoms. This is however not a problem since we can replace all terms that appear in inclusion atoms with new existentially quantified variables.

Hence, for Theorem \ref{hierarkia}, it suffices to prove that $\neg[TC_{\tuple x , \tuple y} \EDGE{k}](\tuple b , \tuple c)$ is not definable in $\inc{k-1}[\tau]$. In this part we will follow the work in \cite{Grohe96}. We first define a simple structure $\A(k,n)$, for $k,n \geq 1$. $\A(k,n)$ consists of two disjoint $E_k$-paths of lenght $n$ i.e.
\begin{align*}
A :=& \{1, \ldots ,n\}\times \{-k, \ldots ,-1,1, \ldots ,k\} \\
E_k^{\A} :=& \{(I,-1), \ldots ,(I,-k)(I+1,-1), \ldots,(I+1,-k)\mid 1 \leq I \leq n\} \\
&\cup \{(I,1), \ldots ,(I,k)(I+1,1), \ldots,(I+1,k)\mid 1 \leq I \leq n\}.
\end{align*}
The following theorem generates a graph $\A$ of the form $\A(k,n)$, for $E_k^{\A}:=\EDGE{k}^{\A}$, with many useful properties. It was originally proved by Grohe using a method of Hrushovski \cite{Hrushovski92} to extend partial isomorphisms of finite graphs.

\begin{theorem}[\cite{Grohe96}]\label{grohethm}
Let $k,n \geq 2$. Then there exists a graph $\A=\A(k,n)$ 
such that:
\begin{enumerate}
\item There exists a mapping $\textrm{row} : A \to \{1, \ldots ,n\}$ such that 
$$\forall a,b \in A:(E^{\A}ab \Rightarrow \row{b} - \row{a} \leq 1).$$
\item There exists an automorphism $\e$ of $\A$ that is self-inverse and preserves the rows i.e.
\begin{itemize}
\item $\e^{-1} = \e$,
\item  $\forall a\in A :\row{\e(a)} = \row{a}$.
\end{itemize}

\item There exist tuples $\tuple b ,\tuple c \in A^k$ in the first and last row respectively (i.e. $\forall i \leq k :(\row{b_i}=1 \ja \row{c_i}=n)$) such that 
\begin{equation*} \A \models \neg[TC_{\tuple x , \tuple y} \EDGE{k}](\tuple b , \tuple c)\textrm{ and } \A \not\models \neg[TC_{\tuple x , \tuple y} \EDGE{k}](\tuple b , \e(\tuple c)).\footnote{In \cite{Grohe96}, $\tuple c$ and $\e (\tuple c)$ are here placed the other way round. This is however not a problem since $\e$ is self-inverse and preserves the rows.}
\end{equation*}
\end{enumerate}

\begin{enumerate}
\setcounter{enumi}{3}
\item For all $a_1,\ldots ,a_{k-1} \in A$ there exists an automorphism $f$ that is self-inverse, preserves the rows, and maps $a_1, \ldots ,a_{k-1}$ according to $\e$, but leaves all elements in rows of distance $>1$ from $\row{a_1}, \ldots ,\row{a_{k-1}}$ fixed i.e.
\begin{itemize}
\item $f^{-1} = f$
\item $\forall a\in A :\row{f(a)} = \row{a}$,
\item $\forall i \leq k-1: f(a_i)=\e(a_i)$,
\item for each $a \in A$ with $\forall i \leq k-1: |\row{a}-\row{a_i}| >1$ we have $f(a) = a$.
\end{itemize}
\end{enumerate}
\end{theorem}

Using this theorem we will prove the following lemma.
\begin{lemma}\label{main}
Let $k \geq 2$ and let $\tau$ be a signature consisting of a binary relation symbol $E$ and $2k$ constant symbols $b_1, \ldots ,b_k,c_1, \ldots ,c_k$. Then $\neg[TC_{\tuple x , \tuple y} \EDGE{k}](\tuple b , \tuple c)$ is not definable in $\inc{k-1}[\tau]$. 
\end{lemma}
The outline of the proof is listed below:

\begin{enumerate}
\item\label{eka}  First we assume to the contrary that there is a $\phi(\tuple b ,\tuple c) \in \inc{k-1}[\tau]$ which is equivalent to $\neg[TC_{\tuple x , \tuple y} \EDGE{k}](\tuple b , \tuple c)$.

\item\label{toka} By Theorem \ref{thm3} we may assume that $\phi$ is of the form $Q^1 x_1 \ldots Q^m
x_m \theta$ where $\theta$ is a quantifier-free $\inc{k-1}[\tau]$ formula. 
\item\label{kolmas} We let $n= 2^{m+2}$ and obtain a graph $\A$ for which items 1-4 of Theorem \ref{grohethm} hold, for $k,n$. In particular, we find tuples $\tuple b$ and $\tuple c$ such that $ \A \models \neg [TC_{\tuple x , \tuple y} \EDGE{k}]  (\tuple b , \tuple c) $ and  $\A \not\models \neg[TC_{\tuple x , \tuple y} \EDGE{k}](\tuple b ,  \e(\tuple c))$. Then by the counter assumption $(\A, \tuple b , \tuple c) \models \phi$ when we find, for $1 \leq i \leq m$, functions 
$$F_i : \{ \emptyset \}[F_1/x_1]\ldots [F_{i-1}/F_{i-1}] \rightarrow \Po(A) \setminus\{\emptyset\}$$ 
such that $F_i(s) = A$ if $Q^i = \forall$, and
$$(\A, \tuple b , \tuple c) \models_X \theta$$
where $X := \{\emptyset\} [F_1/x_1]\ldots [F_{m}/x_{m}]$.

\item\label{neljäs} From $X$ we will construct a team $X^*$ such that 
\begin{equation}\label{holdaa}
(\A, \tuple b , \e(\tuple c)) \models_{X^*} \theta.
\end{equation}
For this, we define operations $\textrm{auto}$ and $\textrm{swap}$ for teams $Y$ of $A$ with $\Dom(Y) = \{x_1, \ldots ,x_m\}$. We first let $\auto{Y}$ be the set of assignments $f \circ s$ where $s \in Y$ and $f$ is a composition of automorphisms of $\A$ that swap $k-1$-tuples of $A$ but leave elements in tuples $\tuple b$ and $\tuple c$ fixed. Then we let $\swap{Y}$ transform each assignment $s$ of $Y$ to a $s'$ such that $s'$ corresponds to Duplicator's choices in a single play of EF$_m((\A, \tuple b , \tuple c) ,(\A, \tuple b , \e(\tuple c)) $ where Spoiler picks members of $(\A, \tuple b , \tuple c)$ according to $s$ and Duplicator picks members of $(\A, \tuple b , \e(\tuple c)) $  according to her winning strategy.
\footnote{A play of this kind is illustrated in Fig. \ref{kuva1}. The idea is that after each round $i\leq m$, $M$ and $N$ are placed so that $N-M \geq 2^{m+1 -i}$. Also for each $j \leq i$, $y_j=x_j$ if $\row{x_j} \leq M$, and $y_j=\e(x_j)$ if $\row{x_j} \geq N$, where $y_j$ and $x_j$ represent Duplicator's and Spoiler's   choices, respectively. In the picture, $\a$ and $\b$ represent two alternative choices Spoiler can make at the fourth round. If Spoiler chooses $x_4:=\a$, then Duplicator chooses $y_4:=\a$, and $M$ is moved to $\row{\a}$. If Spoiler chooses $x_4:=\b$, then Duplicator chooses $y_4:=\e(\b)$, and $N$ is moved to $\row{\b}$. Proceeding in this way we obtain that at the final stage $m$, $(\A, \tuple b , \tuple c, x_1, \ldots,x_m )$ and  $(\A, \tuple b , \e(\tuple c), y_1, \ldots ,y_m) $ agree on all atomic $\FO[\tau]$ formulae.} We then let $X^*:= \swap{\auto{X}}$ and show \eqref{holdaa}.

\item\label{viides} At last, we will show that $X^*$ can be constructed by quantifying $Q^1x_1 \ldots Q^mx_m$ in $\A$ over $\{\emptyset\}$.
\end{enumerate}
Hence we will obtain that $\A \models \phi(\tuple b , \e (\tuple c))$. But now, since $\A \not\models \neg[TC_{\tuple x , \tuple y} \EDGE{k}](\tuple b , \e(\tuple c))$, this contradicts with the assumption that $\phi(\tuple b,\tuple c)$ defines $\neg[TC_{\tuple x , \tuple y} \EDGE{k}](\tuple b , \tuple c)$.

\begin{figure}[h]
\center
\begin{tikzpicture}[xscale=1,yscale=1.3 ]
\draw [->] [ dashed,  thick, shorten >=3pt]  (2.05,3.9) to [out=15,in=165]  (7,3.9) ;
\draw[fill]  (2,3.9) circle [radius=0.03];
\draw[fill] (7,3.9) circle [radius=0.03];

\node [above] (xa) at (2,3.9) {$\b$};

\node [above] (exa) at (7,3.95) {$\e(\b)$};


\draw[fill]  (1.5,2.8) circle [radius=0.03];

\node [left] (xb) at (1.5,2.8){$\a$};

\draw[->, dashed, shorten >=3pt] [thick] (1.5,2.8)  to [my loopss]  (1.5,2.8);

\draw [<->][ultra thick] [red] (4,6) -- (0,6);
\node [above] at (2,6) {$\e(\tuple c)$};
\draw [<->] [ultra thick] [red] (4,0) -- (0,0);
\node [below] at (2,0) {$\tuple b$};

 \draw [<->][ultra thick] [red] (5,6) -- (9,6);
\node [above] at (7,6) {$\tuple c$};

\draw (0,0) -- (0,6);
\draw (4,0) -- (4,6);
\draw  (5,6) -- (5,0) -- (9,0) -- (9,6) ;

\draw [ loosely dotted][very thick] (0,4.7) -- (4,4.7) ;
\draw [loosely dotted][very thick]  (5,4.7) -- (9,4.7) ;
\node [left] at (0,4.7) {$N$};

\draw [ dashed,->, thick] (-0.25,4.55) -- (-0.25,3.9);

\draw [<-, thick, shorten <=3pt]  (3.1,4.7) to [out=15,in=165]  (8,4.7) ;
\draw[fill]  (3.1,4.7) circle [radius=0.03];
\draw[fill] (8,4.7) circle [radius=0.03];

\node [above] (x1) at (3.1,4.75) {$\e(x_{1})$};

\node [above] (ex1) at (8,4.7) {$x_{1}$};

\draw[fill]  (7.5,0.8) circle [radius=0.03];

\node [ left] (xb) at (7.5,0.8){$x_{2}$};

\draw[->, shorten >=3pt] [thick] (7.5,0.8)  to [my loopss]  (7.5,0.8);

\draw[fill]  (2.8,2.2) circle [radius=0.03];

\node [above left] (xb) at (2.8,2.2){$x_{3}$};

\draw[->] [thick, shorten >=3pt] (2.8,2.2)  to [my loopss]  (2.8,2.2);


\draw [loosely dotted][very thick]  (0,2.2) -- (4,2.2) ;
\draw [loosely dotted][very thick]  (5,2.2) -- (9,2.2) ;
\node [left] at (0,2.2) {$M$};

\draw [dashed,->, thick] (-0.25,2.32) -- (-0.25,2.8);

\draw [->, thick] (1,-1) -- (2.95,-1);
\draw [->, thick, dashed ] (1,-1.3) -- (2.95,-1.3);
\draw[fill] (1,-1) circle [radius=0.025];
\draw[fill] (3,-1) circle [radius=0.025];
\draw[fill] (1,-1.3) circle [radius=0.025];
\draw[fill] (3,-1.3) circle [radius=0.025];
\node [above] at (1,-.9) {Spoiler};
\node [above] at (3,-.9) {Duplicator};

\draw [<->, thick] (8,-1.15) --  (8,-.3);
\node [right] at (8.1,-.35) {$n$};
\node [right] at (8.1,-1.1) {$1$};
\draw [<->, thick] (6.3,-1.2) --  (7.9,-1.2);
\node [below] at (6.35,-1.25) {$-k$};
\node [below] at (7.85,-1.25) {$k$};

\end{tikzpicture}
\caption{} \label{kuva1}
\end{figure}
Let us now proceed to the proof.

\begin{proof}[Lemma \ref{main}]
We may start from item \ref{neljäs} of the previous list. Hence we have 
\begin{equation}\label{ekafakta}
(\A, \tuple b , \tuple c) \models_X \theta,
\end{equation}
for $X := \{\emptyset\}[F_1/x_1]\ldots [F_{m}/x_{m}]$, and the first step is to construct a team $X^*$ such that 
\begin{equation}\label{holdaa'}
(\A, \tuple b , \e(\tuple c)) \models_{X^*} \theta.
\end{equation} 
For this, we first define the operation $\textrm{auto}$. By item 4 of Theorem \ref{grohethm}, for all $\tuple a$ listing $a_1,\ldots ,a_{k-1} \in A$ there exists an automorphism $f_{\tuple a}$ which maps $\tuple a$ pointwise to $\e(\tuple a)$, but leaves all elements in rows of distance $>1$ from $\row{a_1}, \ldots ,\row{a_{k-1}}$ fixed.
Let $\mathcal{F}=(F,\circ)$ be the group generated by the automorphisms $f_{\tuple a}$ where $f_{\tuple a}$ is obtained from item 4 of Theorem \ref{grohethm} and $\tuple a$ is a sequence listing $a_1, \ldots ,a_{k-1}\in A$ such that $2 < \row{a_i} < n-1$, for $1 \leq i \leq k-1$. For a team $Y$ of $A$, we then let
$$\auto{Y}:=\{f \circ s \mid f \in \F, s \in Y\}.$$ 

Next we will define the operation $\textrm{swap}$. For this, we will first  define mappings $\textrm{mid}$ and $h$. We let $\textrm{mid}$ map $m$-sequences of  $\{1, \ldots, n\}$ into $\{1, \ldots ,n\}$ so that, for any $\tuple p:=(p_1, \ldots ,p_m)$ and $\tuple q := (q_1, \ldots ,q_m)$ in $\{1, \ldots ,n\}^m$,
\begin{enumerate}
\item $1 < \mi{\tuple p} < n$,
\item $\forall  i \leq m: \mi{\tuple p} \neq p_i$,
\item 
$\forall l \leq n$: if $\tuple p \upharpoonright \{1, \ldots ,l\} = \tuple q \upharpoonright \{1, \ldots ,l\}$, then $\forall  i \leq l: p_i <\mi{\tuple p}$ iff $q_i < \mi{\tuple q}$.
\end{enumerate}
This can be done by following the strategy illustrated in Fig. \ref{kuva1}. We shall explain this in detail in the following. Let $ \tuple p:= (p_1, \ldots ,p_m)$ be a sequence listing natural numbers of $\{1, \ldots ,n\}$. For $\mi{\tuple p}$, we first show how to choose $M$ and $N$, for each $0 \leq i \leq m$, so that 
\begin{itemize}
\item $N-M \geq 2^{m+1 -i}$,
\item $\forall j\leq i : p_j \leq M$ or $p_j \geq N$.
\end{itemize}
We do this inductively as follows. We let $M:= 1$ and $N := n$, for $i=0$. Since $n=2^{m+2}$, clearly the conditions above hold. Assume that $M$ and $N$ are defined for $i$ so that the conditions above hold; we  define $M'$ and $N'$ for $i+1$ as follows:
\begin{enumerate}
\item If $p_{i+1} -M \leq N - p_{i+1}$, then we let $M':= \max\{M,p_{i+1}\}$ and $N':= N$.
\item If $p_{i+1} -M > N - p_{i+1}$, then we let $M':= M $ and $N':=\min\{N, p_{i+1}\}$.
\end{enumerate}
Note that in both cases $\forall j \leq i+1: p_j \leq M'$ or $p_j \geq N'$, and
$$N'-M' \geq \left \lceil{\frac{N-M}{2}}\right \rceil \geq 2^{m+1-(i+1)}.$$
Proceeding in this way we conclude that at the final stage $m$ we have $N-M \geq 2$ with no $p_1, \ldots ,p_m$ stricly in between $M$ and $N$. We then choose $\mi{\tuple p}$ as any number in $]M,N[$. Note that defining  $\textrm{mid}$ in this way we are able to meet the conditions 1-3.

After this we define a mapping $h  : {}^{\{x_1, \ldots ,x_m\}}\!A \to {}^{\{x_1, \ldots ,x_m\}}\!A$. For an assignment $s:\{x_1, \ldots ,x_m\} \rightarrow A$, the assignment $h(s):\{x_1, \ldots ,x_m\} \rightarrow A$ is defined as follows:
$$h(s)(x_i) =
\begin{cases}
 s(x_i) & \text{if }\row{s(x_i)} < \mi{\row{s(\tuple x)}},\\
 \e \circ s(x_i) & \text{if }\row{s(x_i)} > \mi{\row{s(\tuple x)}},
\end{cases}
$$
where $\tuple x := (x_1, \ldots ,x_m)$. For a team $Z$ of $A$ with $\Dom(Z)=\{x_1, \ldots ,x_m\}$, we now let 
$$\swap{Z} :=\{h (s)\mid s\in Z\},$$ 
and define, for each $Y \sub X$,
$$Y^*:=\swap{\auto{Y}}.$$

With $X^*$ now defined, we will next show that \eqref{holdaa'} holds. Without loss of generality we may assume that if a constant symbol $b_j$ (or $c_j$) appears in an atomic subformula $\a$ of $\theta$, then $\a$ is of the form $x_i = b_j$ (or $x_i=c_j$) where $x_i$ is an existentially quantified variable of the quantifier prefix. 
Hence and by \eqref{ekafakta}, it now suffices to show that for all $Y \sub X$ and all quantifier-free $\psi \in \inc{k-1}[\tau]$ with the above restriction for constants,
\begin{equation*}
(\A, \tuple b,\tuple  c) \models_Y \psi \Rightarrow (\A, \tuple b, \e(\tuple c)) \models_{Y^*} \psi.
\end{equation*}
This can be done by induction on the complexity of the quantifier-free $\psi$.
Since $Y^* \cup Z^* = (Y \cup Z)^*$, for $Y,Z \sub X$, it suffices to consider only the case where $\psi$ is an atomic or negated atomic formula. For this, assume that $(\A, \tuple b,\tuple  c) \models_Y \psi$; we will show that 
\begin{equation}\label{wanted}
(\A, \tuple b, \e(\tuple c)) \models_{Y^*} \psi.
\end{equation}
Now $\psi$ is either of the form $x_i = b_j$, $x_i = c_j$, $x_i = x_j$, $\neg x_i = x_j$, $E(x_i,x_j)$, $\neg E(x_i,x_j)$ or $\tuple y \sub \tuple z$ where $b_j,c_j$ are constant symbols and $\tuple y,\tuple z$ are sequences of variables from $\{x_1, \ldots ,x_m\}$ with $|\tuple y| = |\tuple z| \leq k-1$. 
\begin{itemize}
\item Assume first that $\psi$ is of the form $x_i = b_j$ or $x_i = c_j$, and let $s \in Y^*$ be arbitrary. For \eqref{wanted}, it suffices to show by Theorem \ref{flatness} that 
\begin{equation}\label{flattaus}
(\A, \tuple b, \e(\tuple c)) \models_{s} \psi.
\end{equation}
First note that $s = h(f \circ t)$ for some automorphism $f\in \mathcal{F}$ and assignment $t\in Y$ for which, by the assumption and Theorem \ref{flatness}, $(\A, \tuple b, \tuple c) \models_t \psi$. Hence for \eqref{flattaus}, we only need to show that $s(x_i) = t(x_i)$ in case $t(x_i)$ is listed in $\tuple b$, and  $s(x_i)=\e \circ t(x_i)$ in case $t(x_i)$ is listed in $\tuple c$.
For this, first recall that $\mathcal{F}$ is the group generated by automorphisms $f_{\tuple a}$ where $f_{\tuple a}$ is obtained from item 4 of Theorem \ref{grohethm} and $\tuple a$ is a sequence listing $a_1, \ldots ,a_{k-1}\in A$ such that $2 < \row{a_i} < n-1$, for $1 \leq i \leq k-1$. Therefore $f$ leaves all elementes in the first and the last row fixed when $f(\tuple b) = \tuple b$ and $f(\tuple c) = \tuple c$. On the other hand, by the definition of $\rm mid $, $1 <\mi{\row{f\circ t(\tuple x))}} < n$, and hence $h(f \circ t)(x_i)= f\circ t(x_i)$ if $f\circ t(x_i)$ is in the first row, and $h(f \circ t)(x_i)= \e\circ f\circ t(x_i)$ if $f\circ t(x_i)$ is in the last row.
Since tuples $\tuple b$ and $\tuple c$ are in the first and the last row, respectively, we conclude that the claim holds. 
The case where $\psi$ is of the form $x_i = x_j$ or $\neg x_i = x_j$ is straightforward.
\item Assume that $\psi$ is of the form $E(x_i,x_j)$ or $\neg E(x_i,x_j)$. Again, let $s \in Y^*$ when $s=h(f\circ t)$ for some $f \in \F$ and $t\in Y$. For \eqref{flattaus}, consider first the case where
\begin{align}\label{tää}
&\row{t(x_i)},\row{t(x_j)}< \mi{\row{t(\tuple x)}} \textrm{, or }\\
 &\row{t(x_i)},\row{t(x_j)}> \mi{\row{t(\tuple x)}}.\label{toi}
\end{align}
Since $f$ is a row-preserving automorphism, we conclude by the definition of $h$ that $s$ maps both $x_i$ and $x_j$ either according to $f\circ t$ or according to $\e \circ f\circ t$. Since $\e$ is also an automorphism, we obtain \eqref{flattaus} in both cases.
Assume then that \eqref{tää} and \eqref{toi} both fail. Then by symmetry suppose we have
$$\row{t(x_i)}< \mi{\row{t(\tuple x))}} <\row{t(x_j)}.$$
Since $(\A ,\tuple b, \tuple c)\models_t \psi$, we have by item 1 of Theorem \ref{grohethm} that $\psi$ is $\neg E(x_i,x_j)$. Since $f$ and $\e$ preserve the rows, we have
$$\row{s(x_i)}< \mi{\row{s(\tuple x))}} <\row{s(x_j)}.$$
Therefore we obtain $(\A, \tuple b, \tuple c) \models_s \neg E(x_i,x_j)$ which concludes this case. 

\item Assume that $\phi$ is $\tuple y \sub \tuple z$, for some $ \tuple y= y_1 \ldots y_l$ and $\tuple z=z_1 \ldots z_l$ where $l \leq k-1$. Let $s \in Y^*$ be arbitrary. For \eqref{wanted}, we show that there exists a $s' \in Y^*$ such that $s(\tuple y)=s'(\tuple z)$. Now $s= h(f\circ t)$ for some $f\in \mathcal{F}$ and $t \in Y$, and
$(\A, \tuple b, \tuple c) \models_Y \psi$ by the assumption. Hence there exists  a $t' \in Y$ such that $t(\tuple y) = t'(\tuple z)$. Let now $I$ list the indices $1 \leq i \leq l$ for which (i) or (ii) hold:
\footnote{An example where $\tuple y := y_1y_2y_3$ and $\tuple z := z_1z_2z_3$ is illustrated in Fig. \ref{kuva2}. Note that in the example, $I = \{2\}$ since the index number $2$ satisfies (ii). Then letting $s_0 := h(f\circ t)$, we obtain $s(y_1y_3)=s_0(z_1z_3)$ but only $s(y_2) = \e\circ s_0(z_2)$. Fig. \ref{kuva3} shows that choosing $s':=h(f_a \circ f \circ t')$, for $a:= f\circ t'(z_2)$, we obtain $s(\tuple y) = s'(\tuple z)$.}

\begin{align*}\textrm{(i)}&\hspace{2mm}\row{t(y_{i})}<\mi{\row{t(\tuple x)}} \textrm{ and } \row{t'(z_i)} >\mi{\row{t'(\tuple x)}}, 
\\
\textrm{(ii)}&\hspace{2mm}\row{t(y_{i})} > \mi{\row{t(\tuple x)}} \textrm{ and } \row{t'(z_i)} <\mi{\row{t'(\tuple x)}}.
\end{align*}

\begin{figure}[h]
\center

\begin{tikzpicture}[xscale=1.1,yscale=1.3 ]

\draw  (4,6) -- (0,6);
\draw  (4,0) -- (0,0);

 \draw  (5,6) -- (9,6);

\draw (0,0) -- (0,6);
\draw (4,0) -- (4,6);
\draw  (5,6) -- (5,0) -- (9,0) -- (9,6) ;

\draw [ loosely dotted][very thick] (0,4) -- (4,4);
\draw [loosely dotted][very thick]  (5,4) -- (9,4) ;

\draw [loosely dotted][very thick]  (0,2.2) -- (4,2.2) ;
\draw [loosely dotted][very thick]  (5,2.2) -- (9,2.2) ;

\node [right] at (1,-1) {$\begin{cases}
M':=\mi{\row{t' (\tuple x) } } \\
M:= \mi{\row{t (\tuple x) } }
\end{cases}$};

\node [right] at (6,-1) {$\begin{cases}
s= h(f\circ t) \\
s_0:= h(f \circ t')
\end{cases}$};

\draw [<->, very thick, shorten >=.3pt]  (9.3,6) to node [sloped, above] {$s=\e\circ f\circ t$} (9.3,2.2);
\draw [<->, very thick, shorten <=.3pt]  (9.3,2.2) to node [sloped, above] {$s= f\circ t$} (9.3,0);

\draw [<->, very thick, shorten >=.3pt]  (9.8,6) to node [sloped, above] {$s_0=\e \circ f\circ t'$} (9.8,4);
\draw [<->, very thick, shorten <=.3pt]  (9.8,4) to node [sloped, above] {$s_0=  f\circ t'$} (9.8,0);

\node [left] at (0,4) {$M'$ };
\node [left] at (0,2.2) {$M$ };
\draw[fill]  (7.5,5) circle [radius=0.03];
\draw [->, thick,  shorten >=3pt] (7.5,5) to [out=160, in=20] node [sloped, above] {$f$} (6,5);
\node [above right] at (7.5,5) {$t(y_3)$};
\node [below right] at (7.5,5) {$t'(z_3)$};
\draw [->, thick, shorten >=3pt] (6,5) to [out=168,in=12] node [sloped, above] {$\e$} (1.5,5);

\draw[fill]  (6,5) circle [radius=0.03];


\node[above left] at (1.5,5) {$s(y_3) $};
\node[below left] at (1.5,5) {$s_0(z_3) $};

\draw[fill]  (1.5,5) circle [radius=0.03];

\draw[fill]  (1.2,3) circle [radius=0.03];
\draw [->, thick, shorten >=3pt] (1.2,3) to [out=20,in=160] node [sloped, above] {$f$} (2.7,3);
\draw[fill]  (2.7,3) circle [radius=0.03];


\node[above left] at (1.2,3) {$t(y_2)$};
\node[below left] at (1.2,3) {$t'(z_2)$};

\node [below] at (2.7,2.95) {$s_0(z_2)$};

\draw [->, thick, shorten >=3pt] (2.7,3) to [out=15,in=165] node [sloped, above] {$\e$} (6.7,3);

\node [above] at (6.7,3.1) {$s(y_2)$};

\draw[fill]  (6.7,3) circle [radius=0.03];

\draw[fill]  (1.6,1) circle [radius=0.03];
\draw [->, thick, shorten >=3pt] (1.6,1) to [out=20,in=160] node [sloped, above] {$f$} (2.9,1);
\draw[fill]  (2.9,1) circle [radius=0.03];


\node [above left] at (1.6,1) {$t(y_2)$};
\node [below left] at (1.6,1) {$t'(z_2)$};

\node [above right] at (2.9,1) {$s(y_3)$};
\node [below right] at (2.9,1) {$s_0(z_3)$};

\end{tikzpicture}

\caption{} \label{kuva2}
\end{figure}

\begin{figure}[h]
\center
\begin{tikzpicture}[xscale=1.1,yscale=1.3 ]

\draw [<->, very thick, shorten >=.3pt]  (9.3,6) to node [sloped, above] {$s=\e\circ f\circ t$} (9.3,2.2);
\draw [<->, very thick, shorten <=.3pt]  (9.3,2.2) to node [sloped, above] {$s= f\circ t$} (9.3,0);

\draw [<->, very thick, shorten >=.3pt]  (9.8,6) to node [sloped, above] {$s'=\e \circ f_a\circ  f\circ t'$} (9.8,4);
\draw [<->, very thick, shorten <=.3pt]  (9.8,4) to node [sloped, above] {$s'=  f_a\circ f\circ t'$} (9.8,0);

\draw  (4,6) -- (0,6);
\draw  (4,0) -- (0,0);

 \draw  (5,6) -- (9,6);

\draw (0,0) -- (0,6);
\draw (4,0) -- (4,6);
\draw  (5,6) -- (5,0) -- (9,0) -- (9,6) ;

\draw [ loosely dotted][very thick] (0,4) -- (4,4);
\draw [loosely dotted][very thick]  (5,4) -- (9,4) ;

\draw [loosely dotted][very thick]  (0,2.2) -- (4,2.2) ;
\draw [loosely dotted][very thick]  (5,2.2) -- (9,2.2) ;

\node [right] at (1,-1) {$\begin{cases}
M':=\mi{\row{t' (\tuple x) } } \\
M:= \mi{\row{t (\tuple x) } }
\end{cases}$};

\node [right] at (5,-1) {$\begin{cases}
s= h(f\circ t) \\
s':= h(f_a \circ f \circ t') \textrm{, for }a := f\circ t'(z_2)
\end{cases}$};

\node [left] at (0,4) {$M'$ };
\node [left] at (0,2.2) {$M$ };
\draw[fill]  (7.5,5) circle [radius=0.03];
\draw [->, thick,  shorten >=5pt] (7.5,5) to [out=160, in=20] node [sloped, above] {$f$} (6,5);

\draw [->, thick,  shorten >=5pt] (7.5,5) to [out=160, in=20] node [sloped, above] {$f$} (6,5);


\node [above right] at (7.5,5) {$t(y_3)$};
\node [below right] at (7.5,5) {$t'(z_3)$};

\draw[->] [thick, shorten >=4pt] (6,5)  to [my loops] node [sloped, above] {$f_{a}$}  (6,5);

\draw [->, thick, shorten >=3pt] (6,5) to [out=168,in=12] node [sloped, above] {$\e$} (1.5,5);

\draw[fill]  (6,5) circle [radius=0.03];


\node [above left] at (1.5,5) {$s(y_3)$};
\node [below left] at (1.5,5) {$s'(z_3)$};

\draw[fill]  (1.5,5) circle [radius=0.03];

\draw[fill]  (1.2,3) circle [radius=0.03];
\draw [->, thick, shorten >=3pt] (1.2,3) to [out=20,in=160] node [sloped, above] {$f$} (2.7,3);
\draw[fill]  (2.7,3) circle [radius=0.03];


\node [above left] at (1.2,3) {$t(y_2)$};
\node [below left] at (1.2,3) {$t'(z_2)$};


\node [above] at (6.7,3.1) {$s(y_2)$};
\node [below] at (6.7,2.9) {$s'(z_2)$};

\draw [->, thick, shorten >=4pt] (2.7,3) to [out=15,in=165] node [sloped, above] {$\e$} (6.7,3);

\draw [->, thick, shorten >=4pt] (2.7,3) to [out=345,in=195] node [sloped, above] {$f_a$} (6.7,3);

\draw[fill]  (6.7,3) circle [radius=0.03];

\draw[fill]  (1.6,1) circle [radius=0.03];
\draw [->, thick, shorten >=3pt] (1.6,1) to [out=20,in=160] node [sloped, above] {$f$} (2.9,1);
\draw[fill]  (2.9,1) circle [radius=0.03];


\node [above left] at (1.6,1) {$t(y_2)$};
\node [below left] at (1.6,1) {$t'(z_2)$};

\node [above] at (2.9,1.2) {$s(y_3)$};
\node [below] at (2.9,.8) {$s'(z_3)$};

\draw[->] [thick, shorten >=3pt] (2.9,1)  to [my loopss] node [ right] {$f_a$}  (2.9,1);

\end{tikzpicture}
\caption{} \label{kuva3}
\end{figure}

Since $|I| \leq k-1$, choosing $\tuple a := (f\circ t'(z_{i}))_{i \in I}$ we find by item 4 of Theorem \ref{grohethm} an automorphism $f_{\tuple a}$ that swaps $f\circ t'(z_{i})$ to $\e \circ f \circ t'(z_{i})$, for each $i \in I$, but leaves all elementes in rows of distance $>1$ from $ (\row{f\circ t'(z_{i})}  )_{i \in I}$ fixed. We now let $s' := h( f_{\tuple a}\circ f\circ t')$. Since  
$$1 <  \mi{\row{t(\tuple x)}},\mi{\row{t'(\tuple x)}} < n$$
by the definition, we have $2 < \row{t'(z_i)} < n-1$,
for $i \in I$. Hence $f_{\tuple a} \in \F$ and $s' \in Y^*$. Moreover, for $i \in I$, we obtain that
\begin{align*}
\textrm{(i)}&\hspace{2mm}s(y_{i})=f\circ t(y_{i}) =f\circ t'(z_{i}) = \e \circ f_{\tuple a} \circ f\circ t'(z_{i})
=s'(z_{i}), \textrm{ or }\\
\textrm{(ii)}&\hspace{2mm}s(y_{i})=\e\circ f \circ t(y_{i}) = \e\circ f \circ t'(z_{i}) =
f_{\tuple a} \circ f\circ t'(z_{i})=s'(z_{i}).
\end{align*}
For the first and last equalities note that $f_{a}$ and $f$ preserve the rows. For (i) recall also that $\e$ is self-inverse.

Let then $1 \leq j \leq l$ be such that $j \not\in I$ when both (i) and (ii) and fail for $j$. Then we obtain
\begin{align}
&\row{t(y_{j})}>\mi{\row{t(\tuple x)}} \textrm{ and } \row{t'(z_j)} >\mi{\row{t'(\tuple x)}},\textrm{ or }\label{eka}\\
&\row{t(y_{j})}<\mi{\row{t(\tuple x)}} \textrm{ and } \row{t'(z_j)} <\mi{\row{t'(\tuple x)}}. \label{toka}
\end{align}
Assume first that \eqref{eka} holds and let  $i\in I$. Then either
\begin{align*}
\textrm{(i) }&\hspace{2mm}\row{t(y_i)} <\mi{\row{t(\tuple x)}}<\row{t(y_j)} \textrm{, or}\\
\textrm{(ii) }&\hspace{2mm}\row{t'(z_i)} <\mi{\row{t'(\tuple x)}}<\row{t'(z_j)}.
\end{align*}
Since $t(y_j)=t'(z_j)$, $t(y_i)=t'(z_i)$, and $f$ preserves the rows, in both cases we conclude that
$$|\row{f\circ t'(z_{j})} - \row{f\circ t'(z_{i})}| >1.$$
Therefore $f_{\tuple a}$ leaves $f\circ t'(z_j)$ fixed. By \eqref{eka} we now have
$$s(y_{j}) = \e \circ f \circ t (y_j) = \e \circ f \circ t'(z_j)=\e\circ f_{\tuple a} \circ f \circ t'(z_j) = s'(z_j).$$
The case where \eqref{toka} holds is analogous. Hence $s(\tuple y) = s'(\tuple z)$. This concludes the case of inclusion atom and thus the proof of $(\A, \tuple b, \e(\tuple c)) \models_{X^*} \theta$.
\end{itemize}
We have now concluded item \ref{neljäs} of the outline of the proof. Next we will show the last part of the proof. That is, we will show that $X^*$ can be constructed by quantifying $Q^1 x_1 \ldots Q^m x_m$ in $\A$ over $\{\emptyset\}$. For this, it suffices to show the following claim.
\begin{claim}
Let $a \in A$, $p \in \{1, \ldots ,m\}$ be such that $Q^p = \forall$, and $s \in X^* \upharpoonright \{x_1, \ldots ,x_{p-1}\}$. Then $s(a/x_p) \in  X^* \upharpoonright \{x_1, \ldots ,x_p\}$.
\end{claim}
\begin{proof}[Claim]
Let $a$, $p$ and $s$ be as in the assumption. Then 
$$s=h(f \circ t) \upharpoonright  \{x_1, \ldots ,x_{p-1}\},$$
for some $f\in \F$ and $t \in X$. Let $a_0 = f^{-1}(a)$ and $a_1 =f^{-1}\circ \e(a)$. Note that both $t(a_0/x_p) \upharpoonright \{x_1, \ldots ,x_p\}$ and $t(a_1/x_p) \upharpoonright \{x_1, \ldots ,x_p\} $ are in $X \upharpoonright \{x_1, \ldots ,x_p\}$ since $Q^p = \forall$. Let $t_0,t_1 \in X$ extend $t(a_0/x_p) \upharpoonright \{x_1, \ldots ,x_p\}$ and $t(a_1/x_p) \upharpoonright \{x_1, \ldots ,x_p\} $, respectively. It suffices to show that either $h(f \circ t_0)$ or $h(f \circ t_1)$ (which both are in $X^*$) extend $s(a/x_p)$.

First note that since 
$$t_0 \upharpoonright \{x_1, \ldots ,x_{p-1}\} = t_1 \upharpoonright \{x_1, \ldots ,x_{p-1}\} = t \upharpoonright \{x_1, \ldots ,x_{p-1}\}$$ 
we have by item 3 of the definition of $\textrm{mid}$ that, for $i \leq p-1$, inequalities \eqref{1}, \eqref{2} and \eqref{3} are equivalent:
\begin{align}
\row{t_0(x_i)} &< \mi{\row{t_0(\tuple x)}},\label{1}\\
\row{t_1(x_i)} &< \mi{\row{t_1(\tuple x)}},\label{2}\\
\row{t(x_i)} &< \mi{\row{t(\tuple x)}}.\label{3}
\end{align}
Since also $f$ preserves the rows, we have by the definition of $h$ that $h(f \circ t_0)$, $h(f \circ t_1)$ and $h(f \circ t)$ all agree in variables $x_1, \ldots ,x_{p-1}$. Note that also $\e$ preserves the rows, so have $\row{a_0} = \row {a_1}$. Since then $\row{t_0(x_i)} = \row{t_1(x_i)}$, for $i \leq p$, we have by item 3 of the definition of $\textrm{mid}$ that 
$$\row{t_0(x_p)} < \mi{\row{t_0(\tuple x)}} \textrm{ iff }\row{t_1(x_p)} < \mi{\row{t_1(\tuple x)}}.$$
Therefore, either 
$$\row{t_0(x_p)} < \mi{\row{t_0(\tuple x)}} \textrm{ or }\row{t_1(x_p)} > \mi{\row{t_1(\tuple x)}}.$$
Then in the first case $h(f\circ t_0)(x_p) = f\circ t_0(x_p)= a$, and in the second case $h(f \circ t_1)(x_p) = \e \circ f \circ t_1(x_p)= \e\circ \e(a)=a$. Hence $s(a/x_p) \in  X^* \upharpoonright \{x_1, \ldots ,x_p\}$. This concludes the proof of the claim. $\hfill\blacksquare$
\end{proof}
We have now showed that $X^*$ can be constructed by quantifying $Q^1 x_1 \ldots Q^m x_m$ in $\A$ over $\{\emptyset\}$. Also previously we showed that $(\A, \tuple b , \e(\tuple c)) \models_{X^*} \theta$. Therefore, since $\phi = Q^1 x_1 \ldots Q^m x_m \theta$, we obtain that $(\A, \tuple b , \e(\tuple c)) \models \phi$. 
Hence the counter-assumption that $\phi(\tuple b,\tuple c)$ defines $\neg[TC_{\tuple x , \tuple y} \EDGE{k}](\tuple b , \tuple c)$ is false. 
Otherwise $\A \models \neg [TC_{\tuple x , \tuple y} \EDGE{k}]  (\tuple b , \tuple c) $ would yield $(\A, \tuple b , \tuple c) \models \phi$ from which $(\A, \tuple b , \e(\tuple c)) \models \phi$ follows. Therefore we would obtain $\A \models \neg [TC_{\tuple x , \tuple y} \EDGE{k}]  (\tuple b , \e(\tuple c))$ which contradicts with the fact that $\A \not\models \neg [TC_{\tuple x , \tuple y} \EDGE{k}]  (\tuple b , \e(\tuple c))$ by Theorem \ref{grohethm}. This concludes the proof of Lemma \ref{main}.\qed
\end{proof}
Theorem \ref{hierarkia} follows now from Theorem \ref{pietron} and Lemma \ref{main}.

\section{Conclusion}
We have showed that the arity fragments of inclusion logic give rise to an infinite expressivity hierarchy. Earlier, analogous results have been proved for dependence logic and independence logic. We also observed that the $\incforall{k}$-hierarchy collapses at a very low level as it is the case with the $\indforall{k}$-hierarchy. However, the $\depforall{k}$-hierarchy is strict since it can be related to the strict $\ESOfvar{k}$-hierarchy. From the results of \cite{durand12}, \cite{galhankon13} and this article, we obtain the following classification for syntactical hierarchies of dependence, independence and inclusion logic under the lax semantics.
\begin{center}
    \begin{tabular}{ | p{2cm} | p{4,5cm} | p{4,5cm}  |}
    \hline
     & Arity of Dependency Atom & Number of $\forall$  \\ \hline
    $\deplogic$ & strict \newline $\deparity{k}< \newline \deparity{k+1}$ & strict \newline $\depforall{k} < \newline\depforall{2k+2}$  \\[10ex] \hline
    $\indlogic$ & strict \newline $\ind{k} <\newline \ind{k+1}$ & collapse at 2 \newline $\indforall{2} = \indlogic$  \\[10ex] \hline
    $\inclogic$ & \textbf{strict}\newline \textbf{$\inc{k} < \newline\inc{k+1}$} &  \textbf{collapse at 1}\newline \textbf{$\incforall{1} = \inclogic$} \\[10ex]
    \hline
    \end{tabular}
\end{center}

Since $\inclogic$ captures $\PTIME$ in  finite ordered models, it would be interesting to investigate syntactical fragments of inclusion logic in that setting. It appears that then the techniques used in this article would be of no use. Namely, we cannot hope to construct two ordered models in the style of Theorem \ref{grohethm}. 
In fixed point logics, this same question has been studied in the 90s. Imhof showed in \cite{imhof96} that the arity hierarchy of $\PFP$ remains strict in ordered models ($\PFPary{k} <_{\mathcal{O}}\PFPary{k+1}$) by relating the $\PFPary{k}$-fragments to the degree hierarchy within $\PSPACE$. For $\LFP$ and $\IFP$, the same question appears to be more difficult, since both strictness and collapse have strong complexity theoretical consequences.
\begin{theorem}[\cite{imhof96}]
For both $\IFP$ and $\LFP$,
collapse of arity hierarchy in ordered models implies $\PTIME < \PSPACE$, strictness implies $\LOGSPACE < \PTIME$.
\end{theorem}
\begin{proof}
Sketch. For $\IFP$, in case of collapse at $k$, the following chain of (in)equalities can be proved.
$$\PTIME =_{\mathcal{O}}  \IFPary{k} \leq  \PFPary{k} <_{\mathcal{O}} \PFP =_{\mathcal{O}} \rm PSPACE.$$
For IFP, strictness implies
$$\rm LOGSPACE =_{\mathcal{O}} \rm DTC \leq_{\mathcal{O}} \IFPary{1} <_{\mathcal{O}} \IFP =_{\mathcal{O}} \rm PTIME. $$
For $\LFP$ the claim now follows from $\LFPary{k} \leq \IFPary{k}$ and $\IFPary{k} \leq_{\mathcal{O}} \LFPary{2k}$.
\end{proof}
It might be possible to prove similar results for inclusion logic by relating the fragments $\inc{k}$ to arity fragments of fixed point logics. However, the translations between $\inclogic$ and $\GFP$ provided in \cite{gallhella13} do not respect arities. It remains open whether collapse or strictness of the $\inc{k}$-hierarchy have such strong consequences or whether it is possible to relate the $\inc{k}$-fragments in ordered models to the degree hierarchy within $\PTIME$? Another line  would be to find some other syntactical parameter that would fit for this purpose.

\bibliographystyle{plain}
\bibliography{biblio}

\begin{thebibliography}{10}

\bibitem{abramsky09}
Samson Abramsky and Jouko V\"a\"an\"anen.
\newblock From {IF} to {BI}.
\newblock {\em Synthese}, 167:207--230, 2009.
\newblock 10.1007/s11229-008-9415-6.

\bibitem{ajtai83}
Miklos Ajtai.
\newblock {$\Sigma^1_1$}-formulae on finite structures.
\newblock {\em Ann. Pure Appl. Logic}, 4(1):1 -- 48, 1983.

\bibitem{durand12}
Arnaud Durand and Juha Kontinen.
\newblock Hierarchies in dependence logic.
\newblock {\em ACM Transactions on Computational Logic (TOCL)}, 13(4):31, 2012.

\bibitem{galliani12}
Pietro Galliani.
\newblock Inclusion and exclusion dependencies in team semantics: On some
  logics of imperfect information.
\newblock {\em Annals of Pure and Applied Logic}, 163(1):68 -- 84, 2012.

\bibitem{galhankon13}
Pietro Galliani, Miika Hannula, and Juha Kontinen.
\newblock {Hierarchies in independence logic}.
\newblock In Simona Ronchi~Della Rocca, editor, {\em Computer Science Logic
  2013 (CSL 2013)}, volume~23 of {\em Leibniz International Proceedings in
  Informatics (LIPIcs)}, pages 263--280, Dagstuhl, Germany, 2013. Schloss
  Dagstuhl--Leibniz-Zentrum fuer Informatik.

\bibitem{gallhella13}
Pietro Galliani and Lauri Hella.
\newblock {Inclusion Logic and Fixed Point Logic}.
\newblock In Simona Ronchi~Della Rocca, editor, {\em Computer Science Logic
  2013 (CSL 2013)}, volume~23 of {\em Leibniz International Proceedings in
  Informatics (LIPIcs)}, pages 281--295, Dagstuhl, Germany, 2013. Schloss
  Dagstuhl--Leibniz-Zentrum fuer Informatik.

\bibitem{vaananen13}
Pietro Galliani and Jouko~A. V{\"a}{\"a}n{\"a}nen.
\newblock On dependence logic.
\newblock {\em CoRR}, abs/1305.5948, 2013.

\bibitem{gradel13}
Erich Gr{\"a}del.
\newblock Model-checking games for logics of imperfect information.
\newblock {\em Theor. Comput. Sci.}, 493:2--14, 2013.

\bibitem{gradel10}
Erich Gr\"adel and Jouko V\"a\"an\"anen.
\newblock Dependence and independence.
\newblock {\em Studia Logica}, 101(2):399--410, 2013.

\bibitem{Grohe96}
Martin Grohe.
\newblock Arity hierarchies.
\newblock {\em Ann. Pure Appl. Logic}, 82(2):103--163, 1996.

\bibitem{hankon14}
Miika Hannula and Juha Kontinen.
\newblock Hierarchies in independence and inclusion logic with strict
  semantics.
\newblock {\em Manuscript}, 2014.

\bibitem{hodges97}
Wilfrid Hodges.
\newblock {C}ompositional {S}emantics for a {L}anguage of {I}mperfect
  {I}nformation.
\newblock {\em Journal of the Interest Group in Pure and Applied Logics}, 5
  (4):539--563, 1997.

\bibitem{Hrushovski92}
Ehud Hrushovski.
\newblock Extending partial isomorphisms of graphs.
\newblock {\em Combinatorica}, 12(4):411--416, 1992.

\bibitem{imhof96}
Henrik Imhof.
\newblock Computational aspects of arity hierarchies.
\newblock In Dirk van Dalen and Marc Bezem, editors, {\em CSL}, volume 1258 of
  {\em Lecture Notes in Computer Science}, pages 211--225. Springer, 1996.

\bibitem{kontinen_ja10}
Jarmo Kontinen.
\newblock Coherence and computational complexity of quantifier-free dependence
  logic formulas.
\newblock In Juha Kontinen and Jouko V\"a\"an\"anen, editors, {\em Proceedings
  of {D}ependence and {I}ndependence in {L}ogic}, pages 58--77. ESSLLI 2010,
  2010.

\bibitem{vaananen07}
Jouko V\"a\"an\"anen.
\newblock {\em Dependence Logic}.
\newblock Cambridge University Press, 2007.

\bibitem{yang13}
Fan Yang.
\newblock Expressing {S}econd-order {S}entences in {I}ntuitionistic
  {D}ependence {L}ogic.
\newblock {\em Studia Logica}, 101(2):323--342, 2013.

\end{thebibliography}
\end{document}